\documentclass{amsart}
\usepackage{mathtools, microtype, mathrsfs, xfrac, hyperref, amssymb, enumerate, xspace}
\usepackage[latin1]{inputenc}
\usepackage{tikz-cd}

\numberwithin{equation}{section}

\numberwithin{subsection}{section}

\allowdisplaybreaks[1]

\newenvironment{enumerate1}
{\begin{enumerate}[\upshape (1)]}
{\end{enumerate}}

\newtheorem*{namedtheorem}{\theoremname}
\newcommand{\theoremname}{testing}

\newtheorem{theorem}{Theorem}[section]
\newtheorem{proposition}[theorem]{Proposition}
\newtheorem{proposition-definition}[theorem]
{Proposition-Definition}

\newtheorem{lemma}[theorem]{Lemma}

\theoremstyle{definition}

\newtheorem{example}[theorem]{Example}

\theoremstyle{remark}

\renewcommand{\mathcal}{\mathscr}

 \newcommand\cB{\mathcal{B}}

\newcommand\cM{\mathcal{M}} 
\newcommand\cO{\mathcal{O}}

 \newcommand\PP{\mathbb{P}}

\newcommand\fM{\mathfrak{M}} \newcommand\frm{\mathfrak{m}}

\newcommand\arr{\ifinner\to\else\longrightarrow\fi}

\newcommand\arrto{\ifinner\mapsto\else\longmapsto\fi}

\newcommand{\xarr}{\xrightarrow}

\newcommand\op{^{\mathrm{op}}}

\newcommand{\eqdef}{\mathrel{\smash{\overset{\mathrm{\scriptscriptstyle def}} =}}}

\renewcommand\th{^\text{th}}
\newcommand\nd{^\text{nd}}

\def\displaytimes_#1{\mathrel{\mathop{\times}\limits_{#1}}}

\def\displayotimes_#1{\mathrel{\mathop{\bigotimes}\limits_{#1}}}

\newcommand\spec{\operatorname{Spec}}

\newcommand{\cat}[1]{(\mathrm{#1})}

\newcommand{\underaut}{\mathop{\underline{\mathrm{Aut}}}\nolimits}

\newlength{\ignora}

\renewcommand{\setminus}{\smallsetminus}

\newcommand{\mmu}{\boldsymbol{\mu}}

\newcommand\radice[2][\relax]{\hspace{-1.5pt}\sqrt[\uproot{2}#1]{#2}}

\DeclareFontFamily{U}{mathx}{\hyphenchar\font45}
\DeclareFontShape{U}{mathx}{m}{n}{
      <5> <6> <7> <8> <9> <10>
      <10.95> <12> <14.4> <17.28> <20.74> <24.88>
      mathx10
      }{}
\DeclareSymbolFont{mathx}{U}{mathx}{m}{n}
\DeclareFontSubstitution{U}{mathx}{m}{n}
\DeclareMathAccent{\widecheck}{0}{mathx}{"71}
\DeclareMathAccent{\wideparen}{0}{mathx}{"75}

\renewcommand{\epsilon}{\varepsilon}
\newcommand{\cha}{\operatorname{char}}

\newcommand{\ed}{\operatorname{ed}}
\newcommand{\ged}{\operatorname{g\mspace{1mu}ed}}

\newcommand{\trdeg}{\operatorname{tr\mspace{1mu}deg}}

\newcommand{\s}{\subseteq}

\title[Genericity for tame stacks]{The genericity theorem for the essential dimension of tame stacks}

\author{Giulio Bresciani}
\address[G. Bresciani]{Scuola Normale Superiore, Centro di Ricerca Matematica Ennio de Giorgi, Collegio Puteano, Piazza dei Cavalieri 3, 56126 Pisa, Italy}
\email{giulio.bresciani@sns.it}

\author{Angelo Vistoli}
\address[A. Vistoli]{Scuola Normale Superiore, Piazza dei Cavalieri 7, 56126 Pisa, Italy}
\email{angelo.vistoli@sns.it}

\thanks{The first author was partially supported by the DFG Priority Program "Homotopy Theory and Algebraic Geometry" SPP 1786, the second from research funds from the Scuola Normale Superiore, Project \texttt{SNS19\_B\_VISTOLI}. The paper is based upon work partially supported by the Swedish Research Council under grant no. 2016-06596 while the second author was in residence at Institut Mittag-Leffler in Djursholm}

\dedicatory{Dedicated with admiration to Herb Clemens on the occasion of his $82\nd$ birthday.}

\begin{document}

\begin{abstract}
	Let $X$ be a regular tame stack. If $X$ is locally of finite type over a field, we prove that the essential dimension of $X$ is equal to its generic essential dimension, this generalizes a previous result of P.~Brosnan, Z.~Reichstein and the second author. Now suppose that $X$ is locally of finite type over a $1$-dimensional noetherian local domain $R$ with fraction field $K$ and residue field $k$. We prove that $\operatorname{ed}_{k}X_{k} \le \operatorname{ed}_{K}X_{K}$ if $X\to \operatorname{Spec} R$ is smooth and $\operatorname{ed}_{k}X_{k} \le \operatorname{ed}_{K}X_{K}+1$ in general.
\end{abstract}

\maketitle

\section{Introduction, and the statement of the main theorem}

Let $k$ be a field, $X \arr \spec k$ an algebraic stack, $\ell$ an extension of $k$, $\xi \in X(\ell)$ an object of $X$ over $\ell$. If $k \subseteq L \subseteq \ell$ is an intermediate extension, we say, very naturally, that $L$ is a \emph{field of definition} of $\xi$ if $\xi$ descends to $L$. The \emph{essential dimension} $\ed_{k}\xi$, which is either a natural number or $+\infty$, is the minimal transcendence degree of a field of definition of $\xi$. If $X$ is of finite type then $\ed_{k}\xi$ is always finite.

This number $\ed_{k}\xi$ is a very natural invariant, which measures, essentially, the number of independent parameters that are needed for defining $\xi$. The essential dimension $\ed_{k}X$ of $X$ is the supremum of the essential dimension of all objects over all extensions of $k$ (if $X$ is empty then $\ed_{k}X$ is $-\infty$). This number is the answer to the question ``how many independent parameters are needed to define the most complicated object of $X$?''. For example, this is a very natural question for the stack $\cM_{g}$ of smooth projective curves of genus~$g$.

Essential dimension was introduced for classifying stacks of finite groups in \cite{buhler-reichstein}, with a rather more geometric language. Since then it has been actively investigated by many mathematicians. It has been studied for classifying stacks of positive dimensional algebraic group starting from \cite{reichstein-ed-algebraic-groups}, and for more general classes of geometric and algebraic objects in \cite{brosnan-reichstein-vistoli1}. See \cite{berhuy-favi-functorial,brosnan-reichstein-vistoli3, reichstein-ed-survey, merkurjev-ed-survey} for an overview of the subject.

Suppose that $X$ is an integral algebraic stack which is locally of finite type over $k$. We can define the \emph{generic essential dimension}  $\ged_{k}X$ (see \cite[Definition~3.3]{brosnan-reichstein-vistoli3}) as the supremum of all $\ed_{k}\xi$ taken over all $\xi \in X(\ell)$ such that the associated morphism $\xi\colon \spec \ell \arr X$ is dominant. For example, if $X$ has finite inertia and $X \arr M$ is its moduli space, then $M$ is an integral scheme over $k$; if $k(M)$ is its field of rational functions, the pullback $X_{k(M)} \arr \spec k(M)$ is a gerbe (the \emph{generic gerbe of $X$}), and
   \[
   \ged_{k}X = \ed_{k(M)}X_{k(M)} + \dim M\,.
   \]
So, if a generic object of $X$ has trivial automorphism group, then
   \[
   \ged_{k}X =  \dim M = \dim X\,.
   \]
   
P.~Brosnan, Z.~Reichstein and the second author proved two results linking $\ed_{k}X$ and $\ed_{k(M)}X_{k(M)}$.

\begin{enumerate1}

\item The genericity theorem (see \cite{brosnan-reichstein-vistoli1, brosnan-reichstein-vistoli-mixed-char, brosnan-reichstein-vistoli3, reichstein-vistoli-genericity}) says that if $X$ is a smooth integral tame Deligne--Mumford stack over $k$, then $\ed_{k}X = \ged_{k}X$. 

\item If $R$ is a DVR with quotient field $K$ and residue field $k$, and $X \arr \spec R$ is a finite tame étale gerbe, then $\ed_{k}X_{k} \leq \ed_{K}X_{K}$. This was proved in \cite{brosnan-reichstein-vistoli3, brosnan-reichstein-vistoli-mixed-char} (the proof in the first paper did not work when $R$ had mixed characteristics).

\end{enumerate1}

(1) plays a pivotal role in the calculation of the essential dimension of many stacks of geometric interest, such as stacks of smooth or stable curves, stacks of principally polarized abelian varieties \cite{brosnan-reichstein-vistoli3}, coherent sheaves on smooth curves \cite{biswas-dhillon-hoffmann-ed-coherent}, quiver representations \cite{scavia-quiver-ed} and polarized K3 surfaces \cite{gao-k3}. 

In this paper we extend both these results to \emph{weakly tame stacks}, in a somewhat more general form.

Tame stacks have been defined by D.~Abramovich, M.~Olsson and the second author in \cite{dan-olsson-vistoli1}. They are stacks with finite inertia, whose automorphism group schemes for objects over a field are linearly reductive. In characteristic~$0$ these are the Deligne--Mumford stacks with finite inertia; but in positive characteristic they are not necessarily Deligne--Mumford, as there exist finite group schemes that are linearly reductive but not reduced, such as $\mmu_{p}$ in characteristic~$p$; and étale finite group schemes are linearly reductive if and only if their order is prime to the characteristic.

We define weakly tame stacks as algebraic stacks, in the sense of \cite{laumon-moret-bailly}, whose automorphism group schemes for objects over a field are finite and linearly reductive, but whose inertia is not necessarily finite. 

\begin{theorem}\label{thm:genericity1}
Let $X$ be a regular integral weakly tame stack, which is locally of finite type over a field $k$. Then
   \[
   \ed_{k}X = \ged_{k}X\,.
   \]
\end{theorem}

\begin{theorem}\label{thm:genericity2}
   Let $S$ be a regular integral scheme, $X$ an integral, weakly tame stack, $X \arr S$ a smooth morphism. If $s\in S$ is a point with residue field $k(s)$, then
   \[
   \ed_{k(s)} X_{k(s)}\le\ged_{k(S)} X_{k(S)}\,.
   \]
\end{theorem}

\begin{theorem}\label{thm:genericity3}
Let $X$ be a regular integral weakly tame stack, which is locally of finite type over a noetherian $1$-dimensional local domain $R$ with fraction field $K$ and residue field $k$. Then
   \[
   \ed_{k}X_{k} \le \ged_{K}X_{K}+1\,.
   \]
\end{theorem}

In fact, all these theorems have a more general formulation, which applies to individual objects, rather than whole stacks: see Section~\ref{sec:tame} for a discussion, and an application to an interesting geometric case, reduced local intersection curves. We only state this for Theorem~\ref{thm:genericity1}: see Theorem~\ref{thm:genericity4}.

Theorem~\ref{thm:genericity1} generalizes (1) above. Theorem~\ref{thm:genericity2} generalizes (2) above; not because $S$ is not assumed to be the spectrum of a DVR (in fact, the general case easily reduces to this), but mostly because $X$ is not assumed to be a gerbe over $S$. Theorem~\ref{thm:genericity3} is, as far as we know, the first result comparing essential dimensions in mixed characteristic without smoothness assumptions.

Theorem~\ref{thm:genericity2} also implies a slightly weaker version of Theorem~\ref{thm:genericity1}, in which the morphism $X \arr \spec k$ is assumed to be smooth.

Theorems \ref{thm:genericity2} and \ref{thm:genericity3} are related in spirit with the results in \cite{reichstein-scavia-specialization, reichstein-scavia-specialization-II}.

Our approach to the proof of the genericity theorem is very different from those in the references above. The main tool is an new version of the valuative criterion for properness of morphisms of tame stacks (Theorem~\ref{thm:valuative}), which was proved, in a slightly weaker form, by the first author in \cite{giulio-section-birational} in characteristic~$0$. The proof of the general case will appear in \cite{giulio-angelo-lang-nishimura}.

\section{Tame objects and weakly tame stacks}\label{sec:tame}

Here, and for the rest of the paper, \emph{algebraic stacks} will be defined as in \cite{laumon-moret-bailly}; that is, we will assume that the they have finite type diagonal. If $\ell$ is a field, we denote by $(\mathrm{Aff}/\ell)$ the category of affine schemes over $\ell$.

Let $X$ be an algebraic stack, $\ell$ a field, and $\xi$ an object in $X(\ell)$. The functor of automorphisms $\underaut_{\ell}\xi\colon (\mathrm{Aff}/\ell)\op \arr \cat{Grp}$ is a group scheme of finite type. We say that $\xi$ is \emph{tame} if $\underaut_{\ell}\xi\colon (\mathrm{Aff}/\ell)\op \arr \cat{Grp}$ is finite and linearly reductive (see \cite[\S2]{dan-olsson-vistoli1} for a thorough discussion of finite linearly reductive group schemes).

An algebraic stack $X$ is \emph{weakly tame} if every object over a field is tame; it is \emph{tame} if is weakly tame and has finite diagonal (see \cite{dan-olsson-vistoli1}).

The following is a strong form of Theorem~\ref{thm:genericity1}.

\begin{theorem}\label{thm:genericity4} 
Let $X$ be a regular integral stack, which is locally of finite type over a field $k$. If $\ell$ is an extension of $k$ and $\xi$ is a tame object of $X(\ell)$, then
   \[
   \ed_{k}\xi \leq \ged_{k}X\,.
   \]
\end{theorem}

If $X$ has finite diagonal, then there exists an open substack $X^{\mathrm{tame}} \subseteq X$ such that an object $\xi \in X(\ell)$ is tame if an only if $\xi$ is in $X^{\mathrm{tame}}(\ell)$ (see \cite[Proposition~3.6]{dan-olsson-vistoli1}); however, this fails in general if the diagonal of $X$ is only quasi-finite. Thus, for stacks with finite diagonal Theorem~\ref{thm:genericity1} and Theorem~\ref{thm:genericity4} are equivalent, but not in general.

There are many geometrically natural stacks that are Deligne--Mumford in characteristic~$0$, but not in positive characteristic, as the corresponding objects may have infinitesimal automorphisms. Examples include polarized K3 surfaces, surfaces of general type, polarized torsors for abelian varieties, stable maps. In analyzing the essential dimension of these classes of geometric objects in positive characteristic it is essential to go outside the framework of Deligne--Mumford stacks. Here is an example.

Let $k$ be a field, $\ell$ an extension of $k$, $C$ a geometrically reduced, geometrically connected, and local complete intersection curve over $\ell$ of arithmetic genus $g \geq 2$. We are interested in the essential dimension $\ed_{k}C$, that is, the smallest transcendence degree of a field of definition of $C$ over $k$. If $\cha k = 0$ and $C$ has a finite automorphism group, it is proved in \cite[Theorem~7.3]{brosnan-reichstein-vistoli3} that 
\begin{equation}\label{inequality}
   \ed_{k}C \leq 
   \begin{cases}
   3g-3 &\text{if }g \geq 3\\
   5 & \text{if }g = 2
   \end{cases}
\end{equation}

We do not know if this holds in positive characteristic, but we can prove the following. Let us say that $C$ is \emph{tame} if $\underaut_{\ell}C$ is finite and tame. If $\cha \ell = 0$, then the tame condition is automatic.

\begin{proposition}
The inequality \ref{inequality} above holds in any characteristic, if $C$ is tame.
\end{proposition}

\begin{proof}
Let $\fM_{g}^{\mathrm{fin}}$ be the stack over $\spec k$ whose objects over a scheme $S$ are proper flat finitely presented morphisms $C \arr S$, whose fibers are geometrically reduced, geometrically connected, and local complete intersection curves of arithmetic genus $g$. By standard results in deformation theory, this is an integral smooth algebraic stack over $\spec k$. The integer $\ged_{k}\fM_{g}^{\mathrm{fin}} = \ged_{k}\cM_{g}$ is computed in \cite{brosnan-reichstein-vistoli3}; it is equal to $3g-3$ when $g \geq 3$, while it is $5$ when $g = 2$. Thus the result follows from Theorem~\ref{thm:genericity4}.
\end{proof}

One can show that the tame points of $\fM_{g}^{\mathrm{fin}}$ do not form an open substack, so Theorem~\ref{thm:genericity1} would not be sufficient for this application.

Here is an example of a tame curve of the type above, whose automorphism group scheme is not reduced.

\begin{example}
Suppose that $\cha k = p > 0$. Choose three rational points on $\PP^{1}_{k}$, say $0$, $1$ and $\infty$, and call $V \simeq \spec k[x]/\bigl((x-1)^{p}\bigr)$ the $p\th$ infinitesimal neighborhood of $1$. Take two copies of $\PP^{1}$, and glue the two copies of $V$. Call $C_{1}$ and $C_{2}$ two smooth curves with trivial automorphism groups and not isomorphic, and glue them to the union of two $\PP^{1}$, as in the picture below.

   \[
   \includegraphics{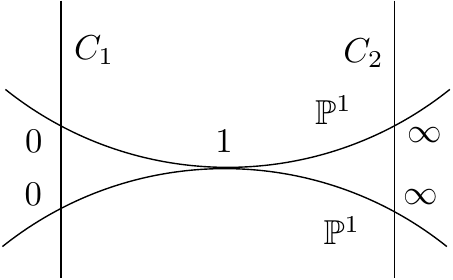}
   \]

\noindent It is an exercise to show that the automorphism group scheme of the resulting curve is $\mmu_{p}$, which is finite and linearly reductive, but not reduced.
\end{example}

\section{Proofs of the theorems}

Let us start with some preliminary results.

\subsection*{A version of the valuative criterion for properness of tame algebraic stacks} 

A basic example of tame stacks is \emph{root stacks} (see \cite[Appendix~B2]{dan-tom-angelo2008}). We will need this in the following situation. Let $R$ be a DVR with uniformizing parameter $\pi$ and residue field $k \eqdef R/(\pi)$. If $n$ is a positive integer, we will denote by $\radice[n]{\spec R}$ the $n\th$ root of the Cartier divisor $\spec k\subseteq \spec R$. It is a stack over $\spec R$, such that given a morphism $\phi\colon T \arr \spec R$, the groupoid of liftings $T \arr \radice[n]{\spec R}$ is equivalent to the groupoid whose objects are triples $(L, s, \alpha)$, where $L$ is an invertible sheaf on $T$, $s \in L(T)$ is a global section of $L$, and $\alpha$ is an isomorphism $L^{\otimes n} \simeq \cO_{T}$, such that $\alpha(s^{\otimes n}) = \phi^\sharp(\pi)$. Alternatively, $\radice[n]{\spec R}$ can be described as the quotient stack $[\spec R[t]/(t^{n} - \pi)/\mmu_{n}]$, where the action of $\mmu_{n}$ on $\spec R[t]/(t^{n} - \pi)$ is by multiplication on $t$. The morphism $\rho\colon \radice[n]{\spec R} \arr \spec R$ is an isomorphism outside of $\spec k \subseteq \spec R$, while the reduced fiber $\rho^{-1}(\spec k)_{\mathrm{red}}$ is non-canonically isomorphic to the classifying stack $\cB_{k}\mmu_{n}$.

Our version of the valuative criterion is as follows.

\begin{theorem}\label{thm:valuative}
Let $f\colon X \arr S$ be a proper morphism where $S$ is a scheme and $X$ a tame stack, $R$ a DVR with quotient field $K$. Suppose that we have a $2$-commutative square
   \[
   \begin{tikzcd}
   \spec K \ar[d, hook]\rar & X \dar{f}\\
   \spec R \rar & S
   \end{tikzcd}
   \]
Then there exists a unique positive integer $n$ and a representable lifting $\radice[n]{\spec R} \arr X$, unique up to a unique isomorphism, of the given morphism $\spec R \arr S$, making the diagram
   \[\begin{tikzcd}
    &   \spec K\rar\ar[dl, hook]   \   &   X\dar   &    \\
      \radice[n]{\spec R}\rar\ar[rru]   &   \spec R\rar\ar[from=u,hook,crossing over]   &   S      & 
   \end{tikzcd}\]
$2$-commutative.
\end{theorem}


The proof of the theorem will appear in \cite{giulio-angelo-lang-nishimura}.

We say that an extension of DVRs $R\subseteq R'$ is \emph{weakly unramified} if its ramification index is $1$; in other words, if a uniformizing parameter of $R$ maps to a uniformizing parameter of $R'$.

\begin{lemma}\label{loopram}
   Let $R \subseteq R'$ be a weakly unramified extension of DVRs, and let $K \subseteq K'$ be the fraction fields of $R$ and $R'$ respectively. Given a diagram 
   \[
   \begin{tikzcd}
   \spec K \ar[d, hook]\rar & X \dar{f}\\
   \spec R \rar & S
   \end{tikzcd}
   \]
with the same hypotheses as in Theorem~\ref{thm:valuative}, if there is an extension $\spec R'\to X$ of $\spec K' \to \spec K \to X$, then there is an extension $\spec R\to X$, too.
\end{lemma}

   \begin{proof}
   By Theorem~\ref{thm:valuative}, there exists a unique positive integer $n$ with a unique representable extension $\radice[n]{\spec R} \to X$, we want to show that $n=1$. Since $R \subseteq R'$ is weakly unramified, then $\radice[n]{\spec R'} \arr \radice[n]{\spec R}$ is representable, hence the composition $\radice[n]{\spec R'} \arr \radice[n]{\spec R} \arr X$ is representable, too. By hypothesis there exists an extension $\spec R' \arr X$, hence the uniqueness part of Theorem~\ref{thm:valuative} implies that $n=1$.
   \end{proof}

\subsection*{Some easy commutative algebra} In the proof of the main theorem we will also use the following well known facts, for which we do not know a reference.

\begin{lemma}\label{lem:precut}
Let $B$ be a regular local ring with $\dim B \ge 1$, and $b \in \frm_{B}$ a non-zero element. There exists a surjective homomorphism $\phi\colon B \arr R$, where $R$ is a DVR, such that $\phi(b) \neq 0$. Furthermore, if $b \in \frm_{B} \setminus\frm_{B}^{2}$ we can find such a $\phi\colon  B \arr R$ with $\phi(b) \in \frm_{R}\setminus\frm_{R}^{2}$.
\end{lemma}

\begin{proof}
If $\dim B = 1$ then we take $B = R$. If $\dim B \geq 2$ we proceed by induction on $\dim B$; it is enough to show that there exists $c \in \frm_{B}\setminus \frm_{B}^{2}$ such that $b \notin (c)$, and $b \notin (c) + \frm_{B}^{2}$ if $b \in \frm_{B} \setminus \frm_{B}^{2}$.

If $b \in \frm_{B}\ \setminus\frm_{B}^{2}$ it is enough to chose $c \in \frm_{B}$ so that the images of $b$ and $c$ in $\frm_{B}/\frm_{B}^{2}$ are linearly independent.

In general, we know that the ring $B$ is a UFD. Let $x$ and $y$ be elements of $\frm_{B}$, whose images in $\frm_{B}/\frm_{B}^{2}$ are linearly independent; the elements $x + y^{d}$, for $d \geq 1$ are irreducible. We claim that they are pairwise not associate: this implies that only finitely many of them can divide $b$, which implies the claim.

To check this, assume that $x + y^{e} = u(x + y^{d})$ for some $u \in B \setminus \frm_{B}$ with $e > d > 0$. Reducing modulo $(y)$ we see that $u \equiv 1 \pmod{y}$, while reducing modulo $(x)$ and dividing by $y^{d}$, which is possible because $A/(x)$ is a domain, we get $u \equiv y^{e-d} \pmod{x}$, and this is clearly a contradiction.
\end{proof}

\begin{lemma}\label{cut}
   Let $A$ be a DVR, $B$ a regular local ring, $A \to B$ a local homomorphism. Assume that the induced ring homomorphism $\frm_{A}/\frm_{A}^{2}\to \frm_{B}/\frm_{B}^{2}$ is injective. There exists a surjective homomorphism $B\to R$ such that $A\to R$ is an injective, weakly unramified extension of DVRs.
\end{lemma}

The proof is immediate from Lemma~\ref{lem:precut}.
   
\begin{lemma}\label{cut2}
Let $U$ be an integral regular scheme, $V \subseteq U$ a nonempty open subset, $u \in U$. There exist a morphism $f\colon \spec R \arr U$, where $R$ is a DVR, which sends the closed point of\/ $\spec R$ to $u$, inducing an isomorphism $k(u) \simeq R/\frm_{R}$, and the generic point into $V$.
\end{lemma}

\begin{proof}
If $u \in V$, this is clear. If not, call $I \subseteq \cO_{U,u}$ the radical ideal of the complement of the inverse image of $V$ in $\spec \cO_{U,u}$, take $b\in I \setminus \{0\}$, and apply Lemma~\ref{lem:precut}.

\end{proof}

\begin{lemma}\label{lem:inequality}
Let $A$ be a noetherian $1$-dimensional local domain with fraction field $K$, $R$ a DVR with quotient field $L$ and residue field $\ell$. Let $A \subseteq R$ be a local embedding, inducing embeddings $k \subseteq \ell$ and $K \subseteq L$. Then
   \[
   \trdeg_{k}\ell \leq \trdeg_{K}L\,.
   \]
\end{lemma}

\begin{proof}
If $A$ is a DVR, this is \cite[Lemma 2.1]{brosnan-reichstein-vistoli-mixed-char} (there it is assumed that $A \subseteq R$ is weakly unramified, but this is not in fact used in the proof). In the general case, consider the normalization $\overline{A}$ of $A$ in $K$, and its localization $A'$ at the maximal ideal $\frm_{R}\cap \overline{A}$. By the Krull--Akizuki theorem $\overline{A}$ is a Dedekind domain, so $A'$ is a DVR.

Call $k'$ the residue field of $A'$; we have a factorization $k\subseteq k' \subseteq \ell$, and $k'$ is algebraic over $k$, so that $\trdeg_{k}\ell = \trdeg_{k'}\ell$, so the general case reduces to the case of an extension of DVRs.
\end{proof}

\begin{lemma}\label{lem:inequality2}
Let $k$ be a field, $R$ a DVR containing $k$, with residue field $\ell$ and fraction field $K$. Assume $\trdeg_{k}\ell < +\infty$. Then
   \[
   \trdeg_{k}\ell < \trdeg_{k}K\,.
   \]
\end{lemma}

\begin{proof}
Suppose $u_{1}$, \dots,~$u_{n}$ are elements of $R$ whose images in $\ell$ are algebraically independent over $k$. If $\pi$ is a uniformizing parameter for $R$, it is immediate to show that $\pi$, $u_{1}$, \dots,~$u_{n}$ are algebraically independent over $k$.
\end{proof}

\subsection*{The proofs}

Let us proceed with the proof of the theorems. For all of them, the first step is to reduce to the case in which $X$ is tame. The reduction is done almost word by word as in the beginning of the proof of \cite[Theorem~6.1]{brosnan-reichstein-vistoli3}. Let $\ell$ be a field, $\xi\colon \spec \ell \arr X$ an object of $X(\ell)$. By a result due essentially to Keel and Mori (\cite[Lemma~6.4]{brosnan-reichstein-vistoli3}) there exists a representable étale morphism $Y \arr X$, where $Y$ is a stack with finite inertia, and a lifting $\eta\colon \spec \ell \arr Y$ of $\xi$. The fact that the morphism $Y \arr X$ is representable implies that the induced homomorphism $\underaut_{\ell}\eta \arr\underaut_{\ell}\xi$ is injective, which implies that $\underaut_{\ell}\eta$ is linearly reductive. By \cite[Proposition~3.6]{dan-olsson-vistoli1} there exists a tame open substack $Y'$ of $Y$ containing $\eta$. The rest of the argument is identical to that in the proof of \cite[Theorem~6.1]{brosnan-reichstein-vistoli3}; so we can assume that $X$ is tame. The theorems are easy consequences of the following.

Let $X \arr S$ be a dominant morphism locally of finite type, where $X$ is a regular integral tame stack, $S$ an integral locally noetherian scheme.
%
%
%
%
%

\begin{lemma}\label{lem:DVR}
Let $\ell$ be a field, $\xi\colon \spec \ell \arr X$ a morphism; call $s \in S$ the image of the composite $\spec \ell \xarr{\xi} X \arr S$. Then $\ell$\/ is an extension of $k(s)$, and we can think of $\xi$ as a morphism $\spec\ell \arr X_{k(s)}$. Assume that $s$ is not the generic point of $S$.
	
\begin{enumerate1} 
\item If $S$ is locally of finite type over a field $k$, there exists a generalization $s'\neq s$ of $s$ such that

		\[
		\ed_{k}\xi \leq \ed_{k}X_{k(s')}\,.
		\]
	
	\item If $S$ has dimension $1$, then
	   \[
	   \ed_{k(s)}\xi \leq  \ed_{k(S)}X_{k(S)}+1\,.
	   \]
	
	\item If $X \arr S$ is smooth, there exists a generalization $s'\neq s$ of $s$ such that
	   \[
	   \ed_{k(s)}\xi \leq\ed_{k(s')}X_{k(s')}\,.
	   \]
\end{enumerate1}
\end{lemma}

\begin{proof}
By \cite[Théorème~6.3]{laumon-moret-bailly} there exists a regular scheme $U$ and a smooth morphism $U\to X$ with a lifting $\spec k\to U$ of $\xi$; if we call $u \in U$ the image of $\spec \ell \arr U$, we can replace $\ell$ with $k(u)$, and assume $\ell = k(u)$. 

In cases (1) and (2), we call $V$ the inverse image of $S \setminus \overline{\{s\}}$ in $U$, and we construct a morphism $\spec R \arr U$ as in Lemma~\ref{cut2} such that the composition $\spec R\to S$ maps the generic point to a generalization $s'\neq s$ of $s$.

In case (3) $X \arr S$ is smooth, so $U \arr S$ is also smooth, and $S$ is regular; in this case we start by choosing a point $s'$ in $S$ such that $s \in \overline{\{s'\}}$, and $\cO_{\overline{\{s'\}},s}$ is a DVR. Then we apply Lemma~\ref{cut} to the embedding $\cO_{\overline{\{s'\}},s} \subseteq \cO_{U',u}$, where $U'$ is the inverse image of $\overline{\{s'\}}$ in $U$, and obtain a morphism $\spec R\to U$ such that the composition $\spec R\to S$ maps the generic point to $s'$ and $\cO_{\overline{\{s'\}},s}\s R$ is weakly unramified.

Let $M$ be the moduli space of $X$, with the resulting factorization $X \arr M \arr S$. Let $K$ be the fraction field of $R$; there exists an intermediate extension $k(s')\s K_{1}\s K$ with
   \[
   \trdeg_{k(s')}K_{1} \le \ed_{k(s')}X_{k(s')}
   \]
and a factorization $\spec K \to \spec K_{1} \to X$. Write $R_{1} = R\cap K_{1}$; clearly $K_{1} \not\subseteq R$, so that $R_{1}$ is a DVR; call $\ell_{1}\s \ell$ its residue field. The composite $\spec R \arr M$ factors through $\spec R_{1}$, and we get a commutative diagram
   \[
   \begin{tikzcd}
   \spec K\rar\dar[hook] & \spec K_{1}\rar & X\dar\ar[dr]\\
   \spec R\ar[urr]\rar       & \spec R_{1}\ar[from=u, hook, crossing over] \rar    & M \rar & S   
   \end{tikzcd}
   \]
By Theorem~\ref{thm:valuative}, since $X\to M$ is proper then there exists an integer $n$ and a representable extension
   \[
   \radice[n]{\spec R_{1}}\to X
   \]
of the morphism $\spec K_{1} \arr X$. Since $X \arr S$ is separated, the diagram
   \[\begin{tikzcd}
   \radice[n]{\spec R}\rar\dar & \radice[n]{\spec R_{1}}\rar  & X\dar\ar[dr]\\
   \spec R\ar[urr]\rar   & \spec R_{1}\ar[from=u, crossing over] \rar    & M \rar &S
   \end{tikzcd}\]
   commutes. If we chose a lifting $\spec \ell \arr \radice[n]{\spec R}$ of the embedding $\spec \ell \subseteq \spec R$ we obtain a factorization
	\[\spec \ell \to \left(\spec \ell_{1}\times_{\spec R_{1}}\radice[n]{\spec R}\right)_{\rm red}= \cB_{\ell_{1}}\mmu_{n} \to X \]
	of $\xi$. Since the essential dimension of $\cB_{\ell_{1}}\mmu_{n}$ over $\ell_{1}$ is at most $1$, we get that
	\[\ed_{k(s)}\xi \leq \trdeg_{k(s)}\ell_{1} +1\,.\]
	If $S$ is locally of finite type over a field $k$, then by Lemma~\ref{lem:inequality2} applied to $R_{1}/k$ we obtain
   \begin{align*}
   \trdeg_{k(s)}\ell_{1}+1 &\leq \trdeg_{k(s')}K_{1} + 1\\
   &\leq \trdeg_{k(s')}K_{1}+\trdeg_{k}k(s')-\trdeg_{k}k(s)\,;
   \end{align*}
Since $\ed_{k}\xi = \ed_{k(s)}\xi + \trdeg_{k}k(s)$ and $\ed_{k}X_{k(s')} = \ed_{k(s')}X_{k(s')} + \trdeg_{k}k(s')$, we obtain the thesis.

	If $S$ has dimension $1$, then $s'$ is the generic point and by Lemma~\ref{lem:inequality} applied to the embedding $\cO_{S,s} \subseteq R_{1}$ we obtain the desired inequality
	\[\trdeg_{k(s)}\ell_{1}+1 \leq \trdeg_{k(S)}K_{1}+1 \leq \ed_{k(S)}X_{k(S)}+1\,.\]

If $X \arr S$ is smooth, then $\cO_{\overline{\{s'\}},s}$ is a DVR and $R$ is weakly unramified over it, and since $\cO_{\overline{\{s'\}},s}\s R_{1}\s R$ then $R$ is weakly unramified over $R_{1}$, too. We can thus assume that $n = 1$, by Lemma~\ref{loopram}, so that $\cB_{\ell_{1}}\mmu_{n} = \spec \ell_{1}$, and by Lemma~\ref{lem:inequality} applied to the embedding $\cO_{\overline{\{s'\}},s}\subseteq R_{1}$ we get
   \[
   \ed_{k(s)}\xi \leq \trdeg_{k(s)}\ell_{1}\leq  \trdeg_{k(s')}K_{1} \leq \ed_{k(s')}X_{k(s')}
   \]
as claimed.
\end{proof}

\subsubsection*{The proof of Theorem~\ref{thm:genericity1}} We apply the above to the case in which $S = M$ is the moduli space of $X$. Let $\xi\colon \spec \ell \arr X$ be a morphism; we need to show that $\ed_{k}\xi \leq\ged_{k}X$. Call $s \in M$ the image of $\xi$; if $s$ is the generic point of $M$, the result follows immediately.

If not, by induction on the codimension of $\overline{\{s\}}$ in $M$ we may assume that the inequality holds for all morphisms $\spec \ell' \arr X$, such that the codimension of the closure of the image $s' \in M$, which equals $\dim M - \trdeg_{k}k(s')$, is less than the codimension of $\overline{\{s\}}$. In particular, this holds for morphisms as above, such that $s'$ is a generalization of $s$ different from $s$.

From Lemma~\ref{lem:DVR} we get a generalization $s'\neq s$ of $s$ and an inequality
   \[
   \ed_{k}\xi \leq \ed_{k}X_{k(s')}\,.
   \]
   By inductive hypothesis, $\ed_{k}X_{k(s')} \leq \ged_{k}X$, hence we conclude.

\subsubsection*{The proof of Theorem~\ref{thm:genericity2}} 
      
By Theorem~\ref{thm:genericity1} applied to $X_{k(S)} \arr \spec k(S)$, it enough to prove that $\ed_{k(s)}X_{k(s)} \leq\ed_{k(S)}X_{k(S)}$ for any $s \in S$. Once again we proceed by induction on the codimension of $\overline{\{s\}}$ in $S$, the case of codimension~$0$ being obvious. If this is positive, given a morphism $\spec\ell\arr X_{k(s)}$, Lemma~\ref{lem:DVR} gives us a generalization $s'$ of $s$ with $\ed_{k(s)}\xi \leq \ed_{k(s')}X_{k(s')}$, and the theorem follows.

\subsubsection*{The proof of Theorem~\ref{thm:genericity3}}

By Theorem~\ref{thm:genericity1} applied to $X_{K} \arr \spec K$, it is enough to prove the inequality $\ed_{k}X_{k} \leq \ed_{K}X_{K}+1$, which follows immediately from Lemma~\ref{lem:DVR}.

\section{Acknowledgments}

We are grateful to the referees for the detailed readings and the constructive comments.

\bibliographystyle{amsalpha}
\bibliography{Genericity}

\end{document}